\documentclass[12pt]{article}
\usepackage{geometry}                
\usepackage{graphicx}
\usepackage{amssymb}
\usepackage{epstopdf}
\newcommand{\hidden}[1]{}

\usepackage{natbib}
\usepackage{amscd}
\usepackage{amsfonts}
\usepackage{amsmath}
\usepackage{amssymb}
\usepackage{amsthm}
\usepackage{cases}		 
\usepackage{cutwin}
\usepackage{enumerate}
\usepackage{graphicx}
\usepackage{hyperref}
\usepackage{ifthen}
\usepackage{lipsum}
\usepackage{mathrsfs}	
\usepackage{mathtools}
\usepackage{multimedia}
\usepackage{url}
\usepackage{wrapfig}
\usepackage{xcolor}

\newcommand*{\A}{\bm{A}}
\newcommand*{\bm}[1]{{\bf #1}}
\newcommand*{\spr}{r}
\newcommand*{\x}{\bm{x}}
\newcommand*{\0}{{\bm 0}}
\newcommand*{\y}{\bm{y}}
\newcommand*{\tr}{^{\!\top}}	
\newcommand*{\B}{\bm{B}}
\newcommand*{\D}{\bm{D}}
\newcommand*{\Complex}{\mathbb{C}} 
\newcommand{\set}[1]{\left\{ \, #1 \, \right\} }
\newcommand*{\suchthat}{\colon}
\newcommand*{\I}{\bm{I}}
\newcommand{\an}[1]{\begin{align}#1\end{align}}	
\newcommand{\ab}[1]{\begin{align*}#1\end{align*}}	
\newcommand*{\p}{\bm{p}}
\newcommand*{\Pc}{\mathcal{P}}
\newcommand*{\eqdef}{\, {:=}\, }
\newcommand*{\Reals}{\mathbb{R}} 
\newcommand*{\z}{\bm{z}}
\newcommand*{\ev}{\bm{e}}
\newcommand*{\diag}[1]{\mbox{ \bf diag}\matrx{#1}}
\newcommand*{\matrx}[1]{{\left[ \stackrel{}{#1}\right]}}
\newcommand*{\uv}{\bm{u}}
\newcommand*{\vv}{\bm{v}}
\newcommand*{\dsfrac}[2]{\frac{\displaystyle #1}{\displaystyle #2} }  
\newcommand*{\xvh}{{\hat{\x}}}
\newcommand*{\yvh}{{\hat{\y}}}
\newcommand*{\xh}{{\hat{x}}}
\newcommand*{\yh}{{\hat{y}}}
\newcommand*{\spb}{s}
\newcommand*{\Realpart}[1]{\mathrm{Re}({#1})}
\newcommand*{\stext}{\shortintertext}	
\renewcommand{\S}{\bm{S}}

\newtheorem{Theorem}{Theorem}
\newtheorem{Corollary}{Corollary}
\title{An Extension of the Rayleigh Quotient to the Spectral Radius of Asymmetric Nonnegative Matrices}
\author {Lee Altenberg
\\
Information and Computer Sciences, \\
University of Hawai`i at M\=anoa,  
{altenber@hawaii.edu}}

\begin{document}
\maketitle

\begin{abstract}
The Rayleigh quotient, which provides the classical variational characterization of the spectral radius of Hermitian matrices, can be extended to nonsymmetric nonnegative irreducible matrices, $\A$, by the inclusion of a diagonal similarity scaling, to yield the variational formula $\spr(\A) = \sup_{\x > \0} \inf_{\y > \0} \x\tr \D_\y \A \D_\y^{-1} \x/(\x\tr \x)$, where $\D_\y$ is the diagonal matrix of the vector $\y$.  Comparison is made to other variational formulae for the spectral radius.
\end{abstract}
The classical variational characterization of the spectral radius of Hermitian matrices is Rayleigh's formula.  Let $\A \in \Complex^{n \times n}$ be an $n \times n$ Hermitian matrix, and let $\spr(\A) = \max \set{| \lambda | \suchthat  \det(\A - \lambda \I) = 0}$ be its spectral radius.   Rayleigh's formula is the maximum of the \emph{Rayleigh quotient} (also known as the \emph{Rayleigh-Ritz ratio}):
\an{
\spr(\A) &= \sup_{\x \neq \0} \frac{\x\tr \A \x^*}{\x\tr\x^*} ,
}
where $\x^*$ is the complex conjugate of $\x$.

No such formula exists for general complex square matrices $\A \in \Complex^{n \times n}$.  We note that in Rayleigh's formula, both left and right sides of the matrix $\A$ enter equivalently, which we refer to here as \emph{structural symmetry}.  

Variational formulae do exist for the class of irreducible nonnegative square matrices.  In contrast with the structurally symmetric form of the Rayleigh quotient, the most well-known formulae for the spectral radius of irreducible nonnegative matrices are \emph{structurally asymmetric}, in that the two sides of the matrix enter in different roles.  There is the classical Collatz-Wielandt formula:
\ab{
\spr(\A) &
= \max_{\x > \0} \min_i \frac{[\A \x]_i}{x_i} 
= \min_{\x > \0} \max_i \frac{[\A \x]_i}{x_i},
}
where only one side of the matrix is involved in vector multiplication (although either side could be chosen), and the Donsker-Varadhan-Friedland formula \citep{Donsker:and:Varadhan:1975,Friedland:1981:Convex}:
\an{
\spr(\A) &
= \sup_{\p \in \Pc} \inf_{\x > \0} \sum_{i=1}^n p_i \frac{[\A \x]_i}{x_i}
= \sup_{\p \in \Pc} \inf_{\x > \0} (\p \circ \x^{-1}){[\A \x]_i},
}
where $\x^{-1} \eqdef [1/x_i]_{i=1}^n$, $\p \circ \x^{-1} = [p_i / x_i]_{i=1}^n$, $\circ$ being the elementwise Hadamard product, and the simplex is represented as $\Pc = \set{\p \suchthat p_i \geq 0, \sum_{i=1}^n p_i = 1} \subset \Reals^n$.  Here the asymmetry enters in the different spaces of variation on the left and right vectors multiplying $\A$.  

\citet{Fiedler:Johnson:Markham:and:Neumann:1985} introduced a structurally symmetric variational formula for the spectral radius: 
\ab{
\spr(\A) &= \max_{\z > \0} \ \min_{\x > \0, \ \y > \0, \ \x \circ \y = \z} \dsfrac{\y\tr \A \x}{\y\tr \x} .
}
Here, both sides of the matrix $\A$ are multiplied by vector variables.  Asymmetries in $\A$ manifest as inequality between $\y$ and $\x$ in the vectors at which the $\min$ and $\max$ are attained.

A change of variables shows that the Donsker-Varadhan-Friedland formula and the Fiedler formula are actually equivalent.  The constraint $\x \circ \y = \z$ in the Fiedler formula means that the degrees of freedom in $\set{\x, \y}$ are not actually $2n$, but only $n$, since there are $n$ constraints $y_i = z_i / x_i$.  An additional $n$ degrees of freedom come from variation of $\z$.  The normalization by $\x\tr \y$ brings the total variational degrees of freedom to $2n - 1$.  With the constraints incorporated, the Fiedler formula becomes
\ab{
\spr(\A) &
= \max_{\z > \0} \ \min_{\x > \0} \dsfrac{(\z \circ \x^{-1}) \tr \A \x}{\z\tr \ev}
= \max_{\z > \0} \ \min_{\x > \0} \sum_{i=1}^n \dsfrac{z_i}{\z\tr \ev}\frac{ [\A \x]_i}{x_i},
}
which is the Donsker-Varadhan-Friedland formula since $\z / {\z\tr \ev} \in \Pc$, which  likewise has $2n-1$ variational degrees of freedom.  The $\inf$ and $\sup$ are replaced by $\min$ and $\max$ because, as \citet{Fiedler:Johnson:Markham:and:Neumann:1985} show, the $\max$ and $\min$ are actually attained for specific $\x> \0, \y > \0$.

In this brief note, we shall see that the methods used by \citet{Fiedler:Johnson:Markham:and:Neumann:1985} to prove Fiedler's formula can be used to prove another structurally symmetric variational formula for the spectral radius.   The formula exhibits the most natural way to handle asymmetries in the matrix, which is through a diagonal similarity scaling.  As in the other variational formulae, it also has $2n-1$ degrees of variational freedom.  What is perhaps most remarkable is that the formula is not already widely known.

Let $\y \in \Reals^n$ and define $\D_\y \eqdef \diag{\y}$ to be the diagonal matrix with diagonal elements $y_i$.  A diagonal similarity scaling of a matrix $\A$ is the product $\D_\y \A \D_y^{-1}$ where we assume $\y > \0$. 

\begin{Theorem}[Structurally Symmetric Variational Formula for the Spectral Radius]\label{testLabel}
\ \\
Let $\A \in \Reals^{n \times n}$ be a nonnegative irreducible $n \times n$ matrix, with Perron root $\spr(\A)$, and Perron vectors $\uv\tr\A = \spr(\A)\uv\tr$ and $\A \vv = \spr(\A)\vv$, normalized so that $\uv\tr \vv = 1$.  Let $\x, \y \in \Reals^n$.  Then
\an{\label{eq:Theorem}
\spr(\A) &= \sup_{\x > \0 } \  \inf_{\y > \0} \frac{  \x\tr \D_\y \A \D_\y^{-1} \x}{\x\tr \x}.
}
Moreover, the $\sup$ and $\inf$ are attained at $\xvh\tr \D_\yvh = \uv \tr$ and $\D_\yvh^{-1} \xvh = \vv$, i.e. $\xh_i \yh_i = u_i$ and $\xh_i / \yh_i = v_i$, hence $u_i v_i = \xh_i^2$, which have solutions
\ab{
\xh_i &= \sqrt{u_i v_i} , \\
\yh_i &= u_i / x_i =  u_i / \sqrt{u_i v_i} = \sqrt{u_i/v_i} \\
&= x_i / v_i 
=  \sqrt{u_i v_i}/ v_i = \sqrt{u_i/v_i}, 
} 
for which we find that
\ab{
\frac{  \xvh\tr \D_\yvh \A \D_\yvh^{-1} \xvh}{\xvh\tr \xvh}
&=  \frac{  \uv\tr \A \vv}{\uv\tr \vv}
= {  \uv\tr \A \vv} = \spr(\A) \uv\tr \vv = \spr(\A) .
}
\end{Theorem}
\begin{proof}
We follow the method of proof of Fiedler's theorem from \citet[pp. 83--86]{Fiedler:Johnson:Markham:and:Neumann:1985}.  First, we show that for any $\x >\0$,  $\inf_{\y > \0} \frac{  \x\tr \D_\y \A \D_\y^{-1} \x}{\x\tr \x} \leq \spr(\A)$.   Given $\x$, let $\y = \uv / \x$.  Then
\ab{
\frac{  \x\tr \D_\y \A \D_\y^{-1} \x}{\x\tr \x}
&= \frac{  \x\tr \D_{\uv/\x} \A \D_{\x/\uv} \x}{\x\tr \x}
= \frac{  \uv\tr \A \D_{\x/\uv} \x}{\x\tr \x}
= \spr(\A) \frac{  \uv\tr  \D_{\x/\uv} \x}{\x\tr \x} \\
&= \spr(\A) \frac{\x\tr \x}{\x\tr\x}
= \spr(\A),
}
hence
\an{
\phi(\x) &\eqdef \inf_{\y > \0} \frac{  \x\tr \D_\y \A \D_\y^{-1} \x}{\x\tr \x} \leq \spr(\A). \label{eq:phiw}
}

Second, we show that $\sup_{\x > \0} \phi(\x) \geq \spr(\A)$.  Consider $\phi(\sqrt{\uv\vv})$:
\ab{
\phi(\sqrt{\uv\vv}) &
= \inf_{\y > \0} \frac{  \sqrt{\uv\vv}\tr \D_\y \A \D_\y^{-1}\sqrt{\uv\vv}}{\sqrt{\uv\vv}\tr \sqrt{\uv\vv}}
= \inf_{\y > \0} \frac{  \sqrt{\uv\vv}\tr \D_\y \A \D_\y^{-1}\sqrt{\uv\vv}}{\uv\tr\vv} \\&
= \inf_{\y > \0}{  \sqrt{\uv\vv}\tr \D_\y \A \D_\y^{-1}\sqrt{\uv\vv}}.
}
Substitute $\z = \y \circ \sqrt{\vv/\uv}$, so $\y = \z \circ \sqrt{\uv/\vv}$.  Then
\ab{
\phi(\sqrt{\uv\vv}) &
= \inf_{\y > \0}{  \sqrt{\uv\vv}\tr \D_\y \A \D_\y^{-1}\sqrt{\uv\vv}}
= \inf_{\z > \0}\sqrt{\uv\vv}\tr \D_{(\z \circ \sqrt{\uv/\vv})} \A \D_{(\z \circ \sqrt{\uv/\vv})}^{-1}\sqrt{\uv\vv} \\&
= \inf_{\z > \0}\z\tr\D_{\uv} \A \D_{\vv}\z^{-1}. 
}
\citet[p. 84]{Fiedler:Johnson:Markham:and:Neumann:1985} and \citet[Corollary 3]{Eaves:Hoffman:Rothblum:and:Schneider:1985} showed that matrices diagonally scaled by their Perron vectors, $\D_{\uv} \A \D_{\vv}$, exhibit for all $\z > \0$,
\an{\label{eq:Fiedler+Eaves}
\z\tr\D_{\uv} \A \D_{\vv}\z^{-1} \geq \ev\tr\D_{\uv} \A \D_{\vv}\ev,
}
hence, $\inf_{\z > \0}\z\tr\D_{\uv} \A \D_{\vv}\z^{-1}$ is attained at $\z = \ev$, which yields
\ab{
\phi(\sqrt{\uv\vv}) &= \inf_{\z > \0}\z\tr\D_{\uv} \A \D_{\vv}\z^{-1} 
= \ev\tr\D_{\uv} \A \D_{\vv}\ev
= \uv\tr \A \vv = \spr(\A).
}
Therefore, 
\an{
\sup_{\x>\0} \phi(\x) &\geq \phi(\sqrt{\uv\vv})  = \spr(\A). \label{eq:supphi}
}
Since by  \eqref{eq:phiw} , $\sup_{\x>\0} \phi(\x) \geq \spr(\A)$ and  by \eqref{eq:supphi} $\phi(\x) \leq \spr(\A)$ for all $\x > \0$, we obtain
\ab{
\sup_{\x>\0} \phi(\x) &=  \sup_{\x > \0 } \  \inf_{\y > \0} \frac{  \x\tr \D_\y \A \D_\y^{-1} \x}{\x\tr \x}
= \spr(\A).
}
\end{proof}
\textbf{Remark:} An interpretation of this result is that the diagonal similarity scaling $\D_{\sqrt{\uv/\vv}} \A \D_{\sqrt{\vv/\uv}}$ renders $\A$ ``as symmetric'' as possible, to the point where it is fit to be used in Rayleigh's quotient.   The vector $\sqrt{\uv/\vv}$ could be thought of as a measure of the asymmetry of $\A$. 

This result is readily extended to the spectral bound $\spb(\A) \eqdef \max \set{ \Realpart{\lambda} \suchthat \det(\A - \lambda \I) = 0}$, of essentially nonnegative matrices, in which off-diagonal elements are nonnegative and diagonal elements are real.
\begin{Corollary}
Let $\A \in \Reals^{n \times n}$ be an irreducible essentially nonnegative $n \times n$ matrix. Let $\x, \y \in \Reals^n$.  Then the spectral bound of $\A$ is
\ab{
\spb(\A) &= \sup_{\x > \0 } \  \inf_{\y > \0} \frac{  \x\tr \D_\y \A \D_\y^{-1} \x}{\x\tr \x}.
}
\end{Corollary}
\begin{proof}
Let $t = \min \set{A_{ii}}$.  Then $\B = \A - t \I \geq \0$, and $\spr(\B) = \spb(\A) - t$.  Hence
\ab{
\spb(\A) &
= \spr(\B) + t 
= \sup_{\x > \0 } \  \inf_{\y > \0} \frac{  \x\tr \D_\y (\A - t \I) \D_\y^{-1} \x}{\x\tr \x} + t\\ 
&=  \sup_{\x > \0 } \  \inf_{\y > \0}\left( \frac{  \x\tr \D_\y \A \D_\y^{-1} \x}{\x\tr \x} -  t   \frac{  \x\tr \D_\y \I \D_\y^{-1} \x}{\x\tr \x} \right) + t \\ &
=  \sup_{\x > \0 } \  \inf_{\y > \0} \frac{  \x\tr \D_\y \A \D_\y^{-1} \x}{\x\tr \x} -t+ t.
}
\end{proof}

Rayleigh's formula is recovered from \eqref{eq:Theorem} in the case of symmetric nonnegative irreducible matrices, $\S$.  We have
\ab{
\phi(\x) &= \inf_{\y > \0} \frac{  \x\tr \D_\y \S \D_\y^{-1} \x}{\x\tr \x}
= \inf_{\y > \0} \frac{  \y\tr \D_\x \S \D_\x \y^{-1}}{\x\tr \x}
}
Since $\D_\x \S \D_\x $ is symmetric, it is \emph{line-sum symmetric}, i.e. $\D_\x \S \D_\x \ev = (\D_\x \S \D_\x) \tr \ev$.  Hence by \eqref {eq:Fiedler+Eaves}, $\inf_{\y > \0} \ev\tr \D_\y (\D_\x \S \D_\x) \D_\y^{-1} \ev $ is attained for $\y = \ev$.  Thus
\ab{
\phi(\x) &= \inf_{\y > \0} \frac{\y\tr \D_\x \S \D_\x \y^{-1}}{\x\tr \x}
= \frac{\ev \tr \D_\x \S \D_\x \ev}{\x\tr \x}
= \frac{\x \tr \S \x }{\x\tr \x}
\stext{so}
\spr(\S) &
= \sup_{\x > \0} \phi(\x) 
=  \sup_{\x > \0}\frac{\x \tr \S \x }{\x\tr \x}
=  \sup_{\x \neq \0}\frac{\x \tr \S \x }{\x\tr \x},
}
which is Rayleigh's formula, where $\x > \0$ can be weakened to $\x \neq \0$ since by Perron-Frobenius theory the maximum is attained for $\x > \0$.

While the Rayleigh quotient can be used to obtain all of the eigenvalues of Hermitian matrices, as is done in the Courant-Fischer theorem \citep[p. 179]{Horn:and:Johnson:1985}, no such direct use of \eqref{eq:Theorem} for nonsymmetric matrices is possible, because $\inf$ and $\sup$ do not apply to the complex eigenvalues that may exist.

\textbf{Conclusion:}  The formula \eqref{eq:Theorem} in Theorem \ref{testLabel} wraps the matrix $\A$ in a diagonal similarity scaling $\D_\y \A \D_\y^{-1}$, which is able to  accommodate any asymmetries in $\A$, and when the rescaled product is placed into the Rayleigh quotient, its maximization produces the spectral radius.  The formula applies to irreducible nonnegative matrices, but much of Perron-Frobenius theory has been extended beyond this class of matrices.  It remains to be addressed just how widely \eqref{eq:Theorem} remains true for these extended matrix classes, and how it may be applied to resolvent positive operators more generally.

\section*{Acknowledgments}
The author thanks Marcus W. Feldman for support from the Stanford Center for Computational, Evolutionary and Human Genomics and the Morrison Institute for Population and Resources Studies, Stanford University, and the Mathematical Biosciences Institute at The Ohio State University, for its support through National Science Foundation Award \#DMS 0931642.  This work was motivated by pursuit of a conjecture by Joel E. Cohen on Levinger's Theorem.


\end{document}